\documentclass[12pt,a4paper]{article}
\usepackage{ejorsj-s2}
\usepackage{amsmath,amssymb,amsfonts,amsthm,url,booktabs}
\usepackage{enumerate}
\usepackage{graphicx}
\usepackage[lined,linesnumbered,boxed]{algorithm2e}
\usepackage{multirow}
\eqswitchon 
\newtheorem{theorem}{Theorem}[section]

\newtheorem{proposition}[theorem]{Proposition}

\newcommand{\LOP}{{\rm LOP}}

\newcommand{\interior}[1]{%
  {\kern0pt#1}^{\mathrm{o}}%
}
\DeclareMathOperator{\order}{O}
\DeclareMathOperator{\ins}{insert}
\SetKwProg{Fn}{Function}{}{end}
\SetKwInput{KwIn}{Input}
\SetKwInput{KwOut}{Output}
\SetKwRepeat{Do}{do}{while}
\SetAlCapSkip{0.5em}
\SetAlCapSty{}
\title{DECOMPOSITION-BASED CONSTRUCTIVE HEURISTICS FOR THE LINEAR ORDERING PROBLEM}
\author{
\begin{tabular}[h]{ccc}  
Kazutoshi Ando& Tatsuya Sugimoto& Noriyoshi Sukegawa \\ % author names
\multicolumn{2}{c}{\textit{Shizuoka University}}
&\multicolumn{1}{c}{\textit{Hosei University}}
\\ % affiliations
\end{tabular}
}
\date{June 23, 2026}
\begin{document}
\maketitle
%\subjclass[2020]{Primary 90C27; Secondary 90C59.}
\begin{abstract}
The linear ordering problem (LOP) is a classical NP-hard combinatorial optimization problem.
In this paper, we study constructive heuristics for the LOP.
We propose a unified recursive framework in which the index set is heuristically partitioned into smaller subsets, sufficiently small subproblems are solved exactly, and the resulting permutations are concatenated.
Within this framework, we investigate three different recursive partitioning principles:
a rule based on strongly connected components (the level graph method),
a cut-based rule (the minimum cut method),
and a score-based rule (the recursive Borda method). 
We first show that the recursive Borda method satisfies the Condorcet criterion, whereas the Borda rule does not.
The level graph method and the minimum cut method satisfy a stronger Condorcet-type property, which the recursive Borda method does not satisfy.
Computational experiments on xLOLIB instances show that the recursive Borda method provides the best average solution quality among the tested constructive heuristics. 
They also show that, after applying insertion-based local search, the differences in final solution quality among these heuristics become much smaller. 
\end{abstract}
%-------- keywords (2 to 6 keywords) ------- 
%%%%%%%%%%%%%%%%%%%%%%%%%%%%%%%%%%%%%%%%%%%%%%%%
%%%%%%%%%%%%%%%%%%%%%%%%%%%%%%%%%%%%%%%%%%%%%%%%%%%%%%%%%%%%%%%%%%%%%%%%%%%%%%%%%%%%%%%%%%%%%%%%%%%%
\section{Introduction}\label{sec:introduction}
%%%%%%%%%%%%%%%%%%%%%%%%%%%%%%%%%%%%%%%%%%%%%%%%%%%%%%%%%%%%%%%%%%%%%%%%%%%%%%%%%%%%%%%%%%%%%%%%%%%%
Let $V$ be a finite index set, and let $n=|V|$.
The {\em linear ordering problem} (LOP) is the problem of finding, for a given nonnegative matrix $C=(c_{uv})_{u,v\in V}$, a permutation $\sigma:\{1,\dots,n\}\rightarrow V$ that maximizes
\begin{align}
  \label{equ:lop}
  f(\sigma) =
  \sum_{i=1}^{n-1}\sum_{j=i+1}^{n} c_{\sigma(i)\sigma(j)} .
\end{align}
The LOP is a classical NP-hard combinatorial optimization problem
and has been studied in various contexts, including the triangulation of input-output tables,
one-machine scheduling with precedence constraints, 
and aggregation of individual preferences;
see, e.g., Mart{\'i} and Reinelt~\cite{Marti} and Sakuraba and Yagiura~\cite{Sakuraba}.
Because of its computational difficulty, a variety of exact and heuristic approaches have been proposed, such as integer programming methods, branch-and-bound algorithms, local search methods, and metaheuristics~\cite{Marti}.
\par
Constructive heuristics provide a basic approach to the LOP, since they construct feasible orderings without relying on iterative improvement.
Such methods can be used either as standalone heuristic procedures or as initial-solution generators for improvement procedures such as local search. 
Classical constructive heuristics such as Becker's algorithm~\cite{Becker1967}
and the Borda rule~\cite{Borda1781,Fishburn1976}
are based on aggregate scores that summarize the tendency of each index to precede the others.
However, some instances may contain coarse ordering information, such as groups of indices that should be placed before other groups.
This leads us to consider decomposition-based constructive heuristics, in which the index set is first partitioned into such groups and the elements within each group are then ordered recursively.
\par
In this paper, we study constructive heuristics based on recursive partitioning of the index set.
The basic idea is to divide the original problem into smaller subproblems, solve sufficiently small subproblems exactly, and concatenate the resulting permutations.
Within this heuristic framework, we consider three partitioning principles.
\par 
The first method is the {\em level graph method}.
It constructs threshold graphs from the matrix entries and recursively partitions the index set using strongly connected component decompositions of these graphs.
The second method is the {\em minimum cut method}, which partitions the index set by minimizing the total weight of arcs directed backward with respect to the resulting two-block ordering.
The third method is the {\em recursive Borda method}, which uses Borda scores to separate indices that tend to precede other indices from those that tend to follow them.
Although these three methods are based on different partitioning principles, they share the same recursive structure. 
We also establish Condorcet-type properties of the proposed methods.
The recursive Borda method satisfies the Condorcet criterion, whereas the Borda rule does not.
The level graph method and the minimum cut method satisfy a stronger Condorcet-type property, which the recursive Borda method does not satisfy.
\par
We evaluate the proposed methods through computational experiments on benchmark instances from xLOLIB~\cite{xLOLIB}.
We first compare the constructive heuristics as standalone methods in terms of solution quality and computation time.
We then examine how different constructive heuristics behave when their solutions are used as initial solutions for insertion-based local search.
\par
The main contributions of this paper are summarized as follows.
\begin{itemize}
\item We propose a unified recursive framework for constructing solutions to the LOP and introduce three recursive partitioning principles within this framework: the level graph method, the minimum cut method, and the recursive Borda method.
\item We establish Condorcet-type properties of the proposed recursive methods. The level graph method and the minimum cut method satisfy a stronger Condorcet-type property, while the recursive Borda method satisfies the Condorcet criterion.
\item We conduct computational experiments on xLOLIB benchmark instances to compare several constructive heuristics, including the proposed methods, in terms of solution quality and computation time. We also examine how these heuristics behave when their solutions are used as initial solutions for local search.
\end{itemize}

The remainder of this paper is organized as follows.
In Section~2, we review existing constructive heuristics and local search methods for the LOP.
In Section~3, we describe the proposed recursive partitioning heuristics and their Condorcet-type properties.
In Section~4, we present the results of computational experiments.
Finally, Section~5 gives concluding remarks. 
%%%%%%%%%%%%%%%%%%%%%%%%%%%%%%%%%%%%%%%%%%%%%%%%%%%%%%%%%%%%%%%%%%%%%%%%%%%%%%%%%%%%%%%%%%%%%%%%%%%%
\section{Preliminaries and Existing Heuristics}\label{sec:preliminaries}
%%%%%%%%%%%%%%%%%%%%%%%%%%%%%%%%%%%%%%%%%%%%%%%%%%%%%%%%%%%%%%%%%%%%%%%%%%%%%%%%%%%%%%%%%%%%%%%%%%%%
In this section, we review baseline constructive heuristics and the insertion-neighborhood local search used in our computational experiments.
%%%%%%%%%%%%%%%%%%%%%%%%%%%%%%%%%%%%%%%%%%%%%%%%%%%%%%%%%%%%%%%%%%%%%%%%%%%%%%%%%%%%%%%%%%%%%%%%%%%%
\subsection{Constructive Heuristics}
\label{subsec:constructive}
The constructive heuristics reviewed in this subsection generate a permutation using aggregate score-based rules.
We first recall Becker's algorithm~\cite{Becker1967}.
For a nonnegative matrix $C=(c_{uv})_{u,v\in V}$, define the score $q(u)$ by
\begin{align}
  \label{equ:cost_q}
  q(u)=
  \frac{\sum\{c_{uv}\mid v\in V,\ v\neq u\}}
       {\sum\{c_{vu}\mid v\in V,\ v\neq u\}}
\end{align}
for each $u\in V$.
Becker's algorithm starts with an empty sequence and repeatedly selects an index $u$ that maximizes the score $q(u)$ and appends it to the sequence.
The selected index and the corresponding row and column are then removed from the matrix, and the scores are recomputed for the remaining indices.
This procedure is repeated until all indices have been selected, and the resulting sequence defines a permutation.
Thus, Becker's algorithm is a sequential score-based construction rule.

The pseudocode of Becker's algorithm is shown in Algorithm~\ref{alg:bec}.
Throughout this paper, ties are broken arbitrarily unless otherwise stated.
\IncMargin{0.5em}
\begin{algorithm}
  \caption{Becker's algorithm.}
  \label{alg:bec}
  \KwIn{Nonnegative matrix $C=(c_{uv})_{u,v\in V}$.}
  \KwOut{Permutation $\sigma:\{1,\dots,n\}\rightarrow V$.}
  \For{$i\leftarrow 1$ {\rm to} $n$}{
    Compute $q(u)$ for all $u\in V$\;
    $u^{*}\leftarrow \arg\max\{q(u)\mid u\in V\}$\;
    $\sigma(i)\leftarrow u^{*}$\;
    Delete the $u^{*}$-th row and the $u^{*}$-th column from $C$\;
    $V\leftarrow V\setminus\{u^{*}\}$\;
  }
  \Return $\sigma$\;
\end{algorithm}
\DecMargin{0.5em}

In implementing Becker's algorithm, special care is required when the denominator of $q(u)$ is zero.
Let
\begin{align}
  p(u)=\sum\{c_{uv}\mid v\in V,\ v\neq u\},
  \quad
  r(u)=\sum\{c_{vu}\mid v\in V,\ v\neq u\}.
\end{align}
We compare $p(u)/r(u)$ and $p(k)/r(k)$ as ordinary ratios when both denominators are positive.
If $r(u)=0$ and $p(u)>0$, then $p(u)/r(u)$ is regarded as larger than any ratio with positive denominator.
If both denominators are zero and both numerators are positive, we compare the numerators.
If $p(u)=r(u)=0$, the index $u$ does not affect the objective value with respect to the remaining indices, and it may be placed at any position.
In our implementation, such an index is placed as early as possible.

A static variant of Becker's algorithm computes $q(u)$ only once at the beginning and outputs the indices in nonincreasing order of $q(u)$.
Although we do not include this static Becker variant in the main computational comparison in Section~\ref{sec:experiments}, it is useful as a conceptual reference point for distinguishing static and sequential score-based construction rules.

Another classical score-based heuristic is the Borda rule~\cite{Borda1781,Fishburn1976}.
For $C=(c_{uv})_{u,v\in V}$, the {\em Borda score} $\beta(u)$ of an index $u$ is defined by
\begin{align}
  \beta(u)
  =
  \sum\{c_{uv}\mid v\in V\}
  -
  \sum\{c_{vu}\mid v\in V\}
\end{align}
for each $u\in V$. This score represents the net preference flow of each index, reflecting how strongly the index dominates other indices in aggregate, and is commonly used in ranking problems. 
The Borda rule is a static ordering rule that outputs the indices in nonincreasing order of $\beta(u)$. 

For conceptual clarity, we also describe a sequential variant of the Borda rule.
This variant starts with an empty sequence and repeatedly selects an index $u$ with maximum Borda score $\beta(u)$ and appends it to the sequence.
The selected index and the corresponding row and column are then removed from the matrix, and the Borda scores are recomputed for the remaining indices.
This procedure is repeated until all indices have been selected, and the resulting sequence defines a permutation.
This variant is not included in the main computational comparison in Section~\ref{sec:experiments}.
The pseudocode of this variant is shown in Algorithm~\ref{alg:borda_v}.

A limitation of these Borda-type rules, in relation to the Condorcet criterion, will be discussed in Section~\ref{subsec:condorcet}.
\IncMargin{0.5em}
\begin{algorithm}[t]
  \caption{Sequential variant of the Borda rule.}
  \label{alg:borda_v}
  \KwIn{Nonnegative matrix $C=(c_{uv})_{u,v\in V}$.}
  \KwOut{Permutation $\sigma:\{1,\dots,n\}\rightarrow V$.}
  \For{$i\leftarrow 1$ {\rm to} $n$}{
    Compute $\beta(u)$ for all $u\in V$\;
    $u^{*}\leftarrow \arg\max\{\beta(u)\mid u\in V\}$\;
    $\sigma(i)\leftarrow u^{*}$\;
    Delete the $u^{*}$-th row and the $u^{*}$-th column from $C$\;
    $V\leftarrow V\setminus\{u^{*}\}$\;
  }
  \Return $\sigma$\;
\end{algorithm}
\DecMargin{0.5em}
\subsection{Local Search}
\label{subsec:local-search}
Local search improves a given permutation by repeatedly replacing it with a better neighboring permutation.
Let $N(\sigma)$ denote the neighborhood of a permutation $\sigma$.
Starting from an initial permutation $\sigma_{0}$, a local search method repeatedly chooses a permutation $\sigma'\in N(\sigma)$ with the largest objective value and replaces $\sigma$ with $\sigma'$ if $f(\sigma')>f(\sigma)$.
The method terminates when
\begin{align}
  \label{equ:ls1}
  f(\sigma)\geq f(\sigma')
  \quad
  \mbox{for all } \sigma'\in N(\sigma).
\end{align}
A permutation satisfying Equation~(\ref{equ:ls1}) is called {\em locally optimal} with respect to $N$.

The basic local search procedure is shown in Algorithm~\ref{alg:ls}.

\IncMargin{0.5em}
\begin{algorithm}[t]
  \caption{Basic local search.}
  \label{alg:ls}
  \KwIn{Nonnegative matrix $C=(c_{uv})_{u,v\in V}$ and initial permutation $\sigma_{0}:\{1,\dots,n\}\rightarrow V$.}
  \KwOut{Locally optimal permutation $\sigma:\{1,\dots,n\}\rightarrow V$ with respect to $N$.}
  $\sigma\leftarrow\sigma_{0}$\;
  \Repeat{no improving neighbor exists}{
    Select $\sigma'\in N(\sigma)$ maximizing $f(\sigma')$\;
    \If{$f(\sigma')>f(\sigma)$}{
      $\sigma\leftarrow\sigma'$\;
    }
  }
  \Return $\sigma$\;
\end{algorithm}
\DecMargin{0.5em}

In this paper, we use the insertion neighborhood $N_I(\sigma)$ as $N(\sigma)$. 
For a permutation $\sigma$ and positions $i \neq j$, the insertion operation removes the element at position $i$ and inserts it into position $j$, shifting the intermediate elements accordingly.
Let $\ins(\sigma,i,j)$ denote the permutation obtained by applying this operation to $\sigma$.
The {\em insertion neighborhood} of $\sigma$ is defined as
\begin{align}
  N_{I}(\sigma)
  =
  \{\ins(\sigma,i,j)\mid i,j=1,\dots,n,\ i\neq j\}.
\end{align}
The change in the objective value can be evaluated efficiently using a standard incremental formula, so that a best improving insertion move can be found in $\order(n^{2})$ time (see, e.g.,~\cite{Marti, LOPRevisited}).
%%%%%%%%%%%%%%%%%%%%%%%%%%%%%%%%%%%%%%%%%%%%%%%%%%%%%%%%%%%%%%%%%%%%%%%%%%%%%%%%%%%%%%%%%%%%%%%%%
\section{Decomposition-Based Constructive Heuristics}\label{sec:heuristics}
%%%%%%%%%%%%%%%%%%%%%%%%%%%%%%%%%%%%%%%%%%%%%%%%%%%%%%%%%%%%%%%%%%%%%%%%%%%%%%%%%%%%%%%%%%%%%%%%%
This section presents three decomposition-based constructive heuristics for the LOP.
These methods share a common recursive framework and differ only in the rule used to partition the current index set.

For a matrix $C=(c_{uv})_{u,v\in V}$, we denote by LOP$(C)$ the linear ordering problem with input matrix $C$.
For a subset $Z \subseteq V$, let $C[Z]$ denote the submatrix of $C$ whose rows and columns are indexed by $Z$.

We first describe the recursive framework common to all the constructive heuristics considered in this section.
The framework recursively constructs a permutation of a current index set $Z\subseteq V$ as follows.
If $|Z|\leq{\tt MAXSIZE}$, the subproblem LOP$(C[Z])$ is solved exactly.
Otherwise, the set $Z$ is partitioned into smaller nonempty subsets $Z_1,\dots,Z_k$ according to a prescribed rule.
The procedure is then applied recursively to each subset $Z_i$, and the resulting permutations are concatenated.
The general procedure is summarized in Algorithm~\ref{alg:general-recursive}.

%\paragraph{\bf Preprocessing convention.}
Before describing the individual partitioning rules, we introduce a preprocessing convention.
For any $(u,v)\in V\times V$, we have
\begin{align}
  c_{uv}
  =
  \max\{c_{uv}-c_{vu},0\}
  +
  \min\{c_{uv},c_{vu}\}.
\end{align}
Therefore, the objective function in Equation~(\ref{equ:lop}) can be rewritten as
\begin{align}
  \label{equ:lop-re}
  f(\sigma)
  =
  \sum_{i=1}^{n-1}\sum_{j=i+1}^{n}
  \max\{c_{\sigma(i)\sigma(j)}-c_{\sigma(j)\sigma(i)},0\}
  +
  \sum_{i=1}^{n-1}\sum_{j=i+1}^{n}
  \min\{c_{\sigma(i)\sigma(j)},c_{\sigma(j)\sigma(i)}\}.
\end{align}
The second term on the right-hand side of Equation~(\ref{equ:lop-re}) is independent of the permutation $\sigma$.
Thus, if $C'=(c'_{uv})_{u,v\in V}$ is defined by
\begin{align}
  \label{equ:c-prime}
  c'_{uv}
  =
  \max\{c_{uv}-c_{vu},0\}
  \quad (u,v\in V),
\end{align}
then LOP($C$) and LOP($C'$) are equivalent.
Hence, in this section, by replacing $C$ with $C'$ when necessary, we assume without loss of generality that
\begin{align}
  \label{equ:asym}
  c_{uv}=0
  \quad \mbox{or} \quad
  c_{vu}=0
  \quad (u,v\in V).
\end{align}

\IncMargin{0.5em}
\begin{algorithm}[t]
  \caption{General recursive constructive heuristic.}
  \label{alg:general-recursive}
  \SetKwFunction{FMain}{GeneralRecursive}
  \Fn{\FMain{$C,V$}}{
    \KwIn{Nonnegative matrix $C=(c_{uv})_{u,v\in V}$ and index set $V$.}
    \KwOut{Permutation $\sigma$ of $V$.}
    \If{$|V|\leq{\tt MAXSIZE}$}{
      \Return an optimal solution of {\rm LOP}$(C)$ obtained by an exact algorithm\;
    }
    \Else{
      $(V_{1},\dots,V_{k})\leftarrow {\tt partition}(C,V)$\;
      \For{$i\leftarrow 1$ {\rm to} $k$}{
        $\sigma^{(i)}\leftarrow \FMain(C[V_i], V_i)$\;
      }
      \Return $(\sigma^{(1)},\dots,\sigma^{(k)})$\;
    }
  }
\end{algorithm}
\DecMargin{0.5em}
%%%%%%%%%%%%%%%%%%%%%%%%%%%%%%%%%%%%%%%%%%%%%%%%%%%%%%%%%%%%%%%%%%%%%%%%%%%%%%%%%%%%%
\subsection{Level Graph Method}
\label{subsec:level}
%%%%%%%%%%%%%%%%%%%%%%%%%%%%%%%%%%%%%%%%%%%%%%%%%%%%%%%%%%%%%%%%%%%%%%%%%%%%%%%%%%%%%
For a nonnegative matrix $C=(c_{uv})_{u,v\in V}$, define the directed graph
$G(C)=(V,A(C))$ 
by
\begin{align}
  A(C)=\{(u,v)\mid u,v\in V,\ c_{uv}>0\}.
\end{align}
This graph is called {\em the strong majority graph} of $C$.

The motivation for the level graph method comes from an exact decomposition property of the LOP based on the strongly connected components of $G(C)$.
Let $\mathcal{P}(C)$ be the partition of $V$ corresponding to the strongly connected components of $G(C)$.
An ordering $(Y_1,\dots,Y_l)$ of $\mathcal{P}(C)$ is called a {\em topological ordering} if
\begin{align}
  \label{equ:isou}
  ((u,v)\in A(C),\ u\in Y_i,\ v\in Y_j,\ i\neq j)
  \Longrightarrow
  i<j.
\end{align}

\begin{theorem}[Strong decomposition theorem~\cite{Hanauer}]
  \label{theo:st_bunkai}
  Let $C$ be an input matrix of the {\rm LOP}, and let
  $(Y_1,\dots,Y_l)$ be a topological ordering of $\mathcal{P}(C)$.
  For each $i=1,\dots,l$, let $C[Y_i]$ denote the restriction of $C$ to $Y_i$.
  If $\sigma^{(i)}:\{1,\dots,|Y_i|\}\rightarrow Y_i$ is an optimal solution of
  {\rm LOP}$(C[Y_i])$, then the concatenated permutation
  \begin{align*}
    \sigma^*=(\sigma^{(1)},\dots,\sigma^{(l)})
  \end{align*}
  is an optimal solution of {\rm LOP}$(C)$.
\end{theorem}

Related structural properties of optimal solutions of the LOP have also been studied in~\cite{AndoSukegawaTakagi}.
Theorem~\ref{theo:st_bunkai} shows that, when the strong majority graph has several strongly connected components, the LOP can be decomposed into independent subproblems without losing optimality.
However, in many instances, $G(C)$ itself is strongly connected or contains a large strongly connected component.
In such cases, the exact decomposition in Theorem~\ref{theo:st_bunkai} does not sufficiently reduce the problem size.

The level graph method is a heuristic extension of this decomposition idea.
It attempts to obtain smaller components by ignoring arcs whose weights are relatively small.
Let the distinct values of $c_{uv}$ $(u,v\in V)$ be
\begin{align}
  0=\gamma_0<\gamma_1<\cdots<\gamma_r.
\end{align}
For $\gamma\geq 0$, {\em the $\gamma$-level graph}
$G_{\gamma}(C)=(V,A_{\gamma}(C))$ 
is defined by
\begin{align}
  \label{equ:level-eda}
  A_{\gamma}(C)
  =
  \{(u,v)\mid u,v\in V,\ c_{uv}>\gamma\}.
\end{align}
When $\gamma=0$, the level graph $G_{\gamma}(C)$ coincides with the strong majority graph $G(C)$.
Thus, the strongly connected component decomposition at level $\gamma=0$ is exactly the decomposition appearing in Theorem~\ref{theo:st_bunkai}.
For $\gamma>0$, some arcs are removed, and the resulting decomposition is no longer guaranteed to preserve optimality.
Nevertheless, it may produce smaller components, and hence it can be used as a heuristic partitioning rule.

The level graph method recursively decomposes the current index set by using strongly connected components of level graphs.
Starting with $\gamma=\gamma_0$, it computes the strongly connected components of $G_{\gamma}(C)$ and orders them topologically.
If a component is larger than ${\tt MAXSIZE}$, the value of $\gamma$ is increased and the same procedure is applied recursively to the submatrix induced by that component.
When all components have size at most ${\tt MAXSIZE}$, the corresponding subproblems are solved exactly and their solutions are concatenated.

The pseudocode is shown in Algorithm~\ref{alg:level}.
Since $\gamma_r$ is the maximum entry value, $G_{\gamma_r}(C)$ has no arcs.
Therefore, all strongly connected components are singletons at this level, and the recursion terminates for ${\tt MAXSIZE}\geq 1$.

\IncMargin{0.5em}
\begin{algorithm}[t]
  \caption{Level graph method.}
  \label{alg:level}
  \SetKwFunction{FMain}{Level}
  \SetKwFunction{FDecom}{Decompose}
  \SetKwProg{Fn}{Function}{}{end}

  \Fn{\FMain{$C,V$}}{
    \KwIn{Nonnegative matrix $C=(c_{uv})_{u,v\in V}$.}
    \KwOut{Permutation of $V$.}
    Let the distinct values of $c_{uv}$ $(u,v\in V)$ be $0=\gamma_0<\gamma_1<\cdots<\gamma_r$\;
    \Return \FDecom{$C,V,0$}\;
  }

  \Fn{\FDecom{$D,Z,t$}}{
    \KwIn{Nonnegative matrix $D=(d_{uv})_{u,v\in Z}$, index set $Z$, and level index $t$.}
    \KwOut{Permutation of $Z$.}
    \If{$|Z|\leq{\tt MAXSIZE}$}{
      \Return an optimal solution of {\rm LOP}$(D)$ obtained by an exact algorithm\;
    }
    \Else{
      Let $(Z_1,\dots,Z_k)$ be a topological ordering of the strongly connected components of $G_{\gamma_t}(D)$\;
      \For{$i\leftarrow 1$ {\rm to} $k$}{
        $\sigma^{(i)} \leftarrow \FDecom(D[Z_i], Z_i, t+1)$\;
      }
      \Return $(\sigma^{(1)},\dots,\sigma^{(k)})$\;
    }
  }
\end{algorithm}
\DecMargin{0.5em}
%%%%%%%%%%%%%%%%%%%%%%%%%%%%%%%%%%%%%%%%%%%%%%%%%%%%%%%%%%%%%%%%%%%%%%%%%%%%%%%%%%%%%
\subsection{Minimum Cut Method}
\label{subsec:mincut}
%%%%%%%%%%%%%%%%%%%%%%%%%%%%%%%%%%%%%%%%%%%%%%%%%%%%%%%%%%%%%%%%%%%%%%%%%%%%%%%%%%%%%
The minimum cut method uses a bipartition of the index set.
Let $G(C)=(V,A(C))$ be the strong majority graph of $C$, where each arc $(u,v)$ has weight $c_{uv}$.
For a bipartition $(V_{1},V_{2})$ of $V$, with $V_{1}\neq\emptyset$, $V_{2}\neq\emptyset$, $V_{1}\cap V_{2}=\emptyset$, and $V_{1}\cup V_{2}=V$, define
\begin{align}
  \label{equ:mincut-value}
  c(V_{1},V_{2})
  =
  \sum\{c_{uv}\mid (u,v)\in A(C),\ u\in V_{2},\ v\in V_{1}\}.
\end{align}
This quantity is the total weight of arcs directed from the second block to the first block.
If the two blocks are ordered as $(V_{1},V_{2})$, such arcs are directed backward with respect to the block ordering.

The minimum cut method chooses a bipartition $(V_{1},V_{2})$ that minimizes $c(V_{1},V_{2})$.
The intuition is that a good two-block ordering should have a small total backward weight.
After obtaining the bipartition, the method recursively solves the subproblems induced by $V_{1}$ and $V_{2}$ and concatenates the resulting permutations.

The pseudocode is shown in Algorithm~\ref{alg:mincut}.

\IncMargin{0.5em}
\begin{algorithm}[t]
  \caption{Minimum cut method.}
  \label{alg:mincut}
  \SetKwFunction{FMain}{MinCut}
  \SetKwProg{Fn}{Function}{}{end}

  \Fn{\FMain{$C,V$}}{
    \KwIn{Nonnegative matrix $C=(c_{uv})_{u,v\in V}$.}
    \KwOut{Permutation of $V$.}
    \If{$|V|\leq{\tt MAXSIZE}$}{
      \Return an optimal solution of {\rm LOP}$(C)$ obtained by an exact algorithm\;
    }
    \Else{
      $(V_{1},V_{2})\leftarrow$ a bipartition of $V$ minimizing $c(V_{1},V_{2})$\;
      $\sigma^{(1)} \leftarrow \FMain(C[V_1], V_1)$\;
      $\sigma^{(2)} \leftarrow \FMain(C[V_2], V_2)$\;
      \Return $(\sigma^{(1)},\sigma^{(2)})$\;
    }
  }
\end{algorithm}
\DecMargin{0.5em}

The minimum cut method is also heuristic.
Minimizing the backward weight between two blocks does not necessarily imply that concatenating optimal solutions of the two subproblems gives an optimal solution of the original problem.
%%%%%%%%%%%%%%%%%%%%%%%%%%%%%%%%%%%%%%%%%%%%%%%%%%%%%%%%%%%%%%%%%%%%%%%%%%%%%%%%%%%%%
\subsection{Recursive Borda Method}
\label{subsec:rborda}
%%%%%%%%%%%%%%%%%%%%%%%%%%%%%%%%%%%%%%%%%%%%%%%%%%%%%%%%%%%%%%%%%%%%%%%%%%%%%%%%%%%%%
The recursive Borda method also uses bipartitions, but its partitioning rule is based on Borda scores.
Recall that the Borda score of $u\in V$ is
\begin{align*}
  \beta(u)
  =
  \sum\{c_{uv}\mid v\in V\}
  -
  \sum\{c_{vu}\mid v\in V\}.
\end{align*}
A positive Borda score indicates that the index $u$ tends to dominate other indices in aggregate, while a negative score indicates the opposite.
Unlike the static Borda ordering and its sequential variant described in Section~\ref{subsec:constructive}, the recursive Borda method uses these scores to define bipartitions recursively. 
\par
For a bipartition $(V_{1},V_{2})$ of $V$, define
\begin{align}\label{equ:max2op}
  b(V_{1},V_{2})
 & = \sum\{c_{uv}\mid (u,v)\in A(C),\ u\in V_{1},\ v\in V_{2}\}\notag\\
 &\quad - \sum\{c_{uv}\mid (u,v)\in A(C),\ u\in V_{2},\ v\in V_{1}\}.
\end{align}
This value represents the net weight supporting the block order $(V_{1},V_{2})$.
The recursive Borda method chooses a bipartition that maximizes $b(V_{1},V_{2})$ and then applies the same procedure recursively to the two induced subproblems.

The following proposition gives a simple characterization of such a bipartition.

\begin{proposition}
  \label{prop:borda-cut}
  Let $(V_{1},V_{2})$ be a bipartition of $V$.
  Then
  \begin{align}
    \label{equ:max2op-f}
    b(V_{1},V_{2})
    =
    \sum_{u\in V_{1}}\beta(u).
  \end{align}
  Moreover,
  \begin{align}
    \label{equ:borda-sum-zero}
    \sum_{u\in V}\beta(u)=0.
  \end{align}
  Hence, if not all Borda scores are zero, then both sets
\begin{align*}
    \{u\in V\mid \beta(u)>0\}
    \quad \mbox{and} \quad
    \{u\in V\mid \beta(u)<0\}
\end{align*}
  are nonempty.
  A bipartition maximizing $b(V_{1},V_{2})$ is obtained by assigning all indices with positive Borda score to $V_{1}$ and all indices with negative Borda score to $V_{2}$.
  Indices with zero Borda score may be assigned to either block without changing the value of $b(V_{1},V_{2})$.
\end{proposition}
\begin{proof}
Expanding $\sum_{u\in V_{1}}\beta(u)$, every term $c_{uv}$ with $u,v\in V_{1}$ appears once with a positive sign and once with a negative sign, and hence cancels out.
Terms with both endpoints in $V_{2}$ do not appear.
A term $c_{uv}$ with $u\in V_{1}$ and $v\in V_{2}$ appears once with a positive sign, while a term $c_{uv}$ with $u\in V_{2}$ and $v\in V_{1}$ appears once with a negative sign.
Therefore, the remaining terms are exactly those in Equation~(\ref{equ:max2op}), and Equation~(\ref{equ:max2op-f}) follows.

Furthermore,
\begin{align*}
  \sum_{u\in V}\beta(u)
  =
  \sum_{u\in V}\sum_{v\in V}c_{uv}
  -
  \sum_{u\in V}\sum_{v\in V}c_{vu}
  =
  0.
\end{align*}
Thus, unless all Borda scores are zero, there must be at least one positive Borda score and at least one negative Borda score.
The characterization of a maximizing bipartition follows immediately from Equation~(\ref{equ:max2op-f}).
\end{proof}

In the recursive Borda method, indices with positive Borda score are placed in the first block, and indices with negative Borda score are placed in the second block.
Indices with zero Borda score are assigned according to a prescribed tie-breaking rule.
In our implementation, zero-score indices are assigned one by one to the currently smaller block, so that the two blocks are kept as balanced as possible.
If all Borda scores are zero, this rule splits the index set into two nonempty subsets whose sizes differ by at most one.

The pseudocode of the recursive Borda method is shown in Algorithm~\ref{alg:rborda}.

\IncMargin{0.5em}
\begin{algorithm}[t]
  \caption{Recursive Borda method.}
  \label{alg:rborda}
  \SetKwFunction{FMain}{RBorda}
  \SetKwProg{Fn}{Function}{}{end}

  \Fn{\FMain{$C,V$}}{
    \KwIn{Nonnegative matrix $C=(c_{uv})_{u,v\in V}$.}
    \KwOut{Permutation of $V$.}
    \If{$|V|\leq{\tt MAXSIZE}$}{
      \Return an optimal solution of {\rm LOP}$(C)$ obtained by an exact algorithm\;
    }
    \Else{
      Compute the Borda score $\beta(u)$ for all $u\in V$\;
      $V_{1}\leftarrow\{u\in V\mid \beta(u)>0\}$\;
      $V_{2}\leftarrow\{u\in V\mid \beta(u)<0\}$\;
      Assign indices with $\beta(u)=0$ to $V_{1}$ or $V_{2}$ according to the prescribed tie-breaking rule\;
      $\sigma^{(1)} \leftarrow \FMain(C[V_1], V_1)$\;
      $\sigma^{(2)} \leftarrow \FMain(C[V_2], V_2)$\;
      \Return $(\sigma^{(1)},\sigma^{(2)})$\;
    }
  }
\end{algorithm}
\DecMargin{0.5em}
\par 
The recursive Borda method differs from the static Borda rule in that Borda scores are recomputed within each recursively generated subset and are used to define the subsequent recursive partition.
%%%%%%%%%%%%%%%%%%%%%%%%%%%%%%%%%%%%%%%%%%%%%%%%%%%%%%%%%%%%%%%%%%%%%%%%%%%%%%%%%%%%%%%
\subsection{Condorcet Properties}\label{subsec:condorcet}
%%%%%%%%%%%%%%%%%%%%%%%%%%%%%%%%%%%%%%%%%%%%%%%%%%%%%%%%%%%%%%%%%%%%%%%%%%%%%%%%%%%%%%%
We next establish the Condorcet-type properties of the proposed decomposition-based heuristics.
\par
Although we use the preprocessing convention in Equation~(\ref{equ:asym}) throughout this section, we state the following definition in its standard form.
An element $x\in V$ is called a {\it Condorcet winner} for a nonnegative matrix
$C=(c_{uv})_{u,v\in V}$ if $c_{xv}>c_{vx}$ for all $v\in V-{x}$.
Under the convention in Equation~(\ref{equ:asym}), this condition is equivalent to $c_{xv}>0$ for all $v\in V-{x}$.
A Condorcet winner is unique if it exists.
It is known that the optimal solutions of the linear ordering problem satisfy 
a property called the {\it Condorcet criterion}~\cite{YL78,Young88}. 
\begin{proposition}[Condorcet criterion]\label{pr:CC}
Let $C=(c_{uv})_{u,v\in V}$ be a nonnegative matrix. If $x\in V$ is the Condorcet winner for $C$, 
then for every optimal solution $\pi$ of $\LOP(C)$, we have $\pi(1) =x$.
\end{proposition}
\par 
The static and sequential Borda rules do not satisfy this criterion.
This can be seen from the following example. 
Let $V=\{x,a,b\}$, and let $C=(c_{uv})_{u,v\in V}$ be defined by
\begin{align*}
c_{xa}=c_{xb}=1, c_{ab}=100, c_{ax}=c_{bx}=c_{ba}=0.
\end{align*}
Then, $x$ is the Condorcet winner of $C$, but both the static and sequential Borda rules do not select $x$ 
as the first element. 

Ando, Sukegawa, and Takagi~\cite{AndoSukegawaTakagi} proposed a stronger property for the optimal solutions of 
the linear ordering problem, 
called the {\it strong Condorcet criterion}. 
We recall that for a nonnegative matrix $C=(c_{uv})_{u,v\in V}$, $\mathcal{P}(C)$ is the partition of $V$ corresponding to the strongly connected 
components of the strong majority graph $G(C)=(V,A(C))$ for $C$.
%%%%%%%%%%%%%%%%%%%%%%%%%%%%%%%%%%%%%%%%%%%%
\begin{proposition}[Ando, Sukegawa, and Takagi~\cite{AndoSukegawaTakagi}]\label{th:SCC}
Let $C=(c_{uv})_{u,v\in V}$ be a nonnegative matrix. 
Every optimal solution $\pi$ of $\LOP(C)$ satisfies the strong Condorcet criterion, 
that is, $\pi$ satisfies
\begin{align}\label{eq:SCC}
\forall Y,Z\in\mathcal{P}(C) \, \mbox{with $Y\not=Z$}, \,  \forall y\in Y, \, \forall z\in Z\colon \,  (y,z)\in A(C) \implies \pi^{-1}(y)<\pi^{-1}(z).
\end{align}
\end{proposition}
\begin{proposition}
\label{pr:condorcet-level}
The level graph method satisfies the strong Condorcet criterion.
\end{proposition}
\begin{proof}
Let $\pi\colon\{1,\ldots,n\}\to V$ be a permutation derived by the execution of the level graph method (Algorithm~\ref{alg:level}) and 
let $Y,Z\in\mathcal{P}(C)$, $Y\not=Z$, $y\in Y$, $z\in Z$, and $(y,z)\in A(C)$. 
If $|V|\leq {\tt MAXSIZE}$, then, by the definition of the algorithm, $\pi$ is an exact solution of $\LOP(C)$ so that~\eqref{eq:SCC} is 
satisfied by Proposition~\ref{th:SCC}. 
Otherwise, the algorithm decomposes the strong majority graph $G(C)=(V,A(C))=G_0(C)$ into strongly connected components $\mathcal{P}(C)$ and 
yields a topological ordering $(Y_1,\ldots, Y_l)$ of $\mathcal{P}(C)$. Since $(y,z)\in A(C)$, we have $Y=Y_i$ and $Z=Y_j$ for some $i<j$ by the definition of 
a topological ordering~\eqref{equ:isou}. It follows from the definition of the algorithm that $y$ precedes $z$ in $\pi$, i.e., $\pi^{-1}(y)<\pi^{-1}(z)$. 
\end{proof}
\begin{proposition}
\label{pr:condorcet-mincut}
The minimum cut method satisfies the strong Condorcet criterion.
\end{proposition}
\begin{proof}
We prove that for every subset $S\subseteq V$ {\tt MinCut}$(C[S],S)$ in Algorithm~\ref{alg:mincut} 
returns a permutation $\pi\colon\{1,\ldots,|S|\}\to S$ satisfying~\eqref{eq:SCC} by induction on $|S|$. 
If $|S|\leq{\tt MAXSIZE}$, then, by the definition of the algorithm, $\pi$ is an exact solution of $\LOP(C[S])$, so that we have~\eqref{eq:SCC} by Proposition~\ref{th:SCC}.  
\par 
Suppose $|S|>{\tt MAXSIZE}$ and assume that  {\tt MinCut}$(C[S'],S')$ returns a permutation $\sigma'$ 
that satisfies the strong Condorcet criterion for all proper subset $S'$ of $S$ with $x\in S'$. 
Let $Y,Z\in\mathcal{P}(C)$, $Y\not=Z$, $y\in Y$, $z\in Z$, and $(y,z)\in A(C)$. If  
$(S_1,S_2)$ is a minimum cut of $G(C[S])$, then there exists a partition $\{\mathcal{P}_1,\mathcal{P}_2\}$ of 
$\mathcal{P}(C)$ such that 
\begin{align}
S_1=\bigcup_{Y'\in\mathcal{P}_1} Y',\qquad S_2=\bigcup_{Z'\in\mathcal{P}_2}Z'.
\end{align} 
If $Y\in\mathcal{P}_1$ and $Z\in\mathcal{P}_2$, then we have $(y,z)\in S_1\times S_2$, and hence, 
$y$ precedes $z$ in $\pi$ by the definition of the algorithm. Otherwise, 
$Y$ and $Z$ are strong components of the strong majority graph $G(C[S_i])$ and $(y,z)\in A(C[S_i])$ 
for some $i=1,2$. It follows from the induction hypothesis that 
$y$ precedes $z$ in $\sigma^{(i)}$ for $i=1$ or $i=2$. Since $\pi=(\sigma^{(1)},\sigma^{(2)})$, 
$y$ precedes $z$ in $\pi$.  
\end{proof}
\par 
In contrast to the static Borda rule, the recursive Borda method satisfies the Condorcet criterion.
\begin{proposition}
\label{pr:condorcet-rborda}
The recursive Borda method satisfies the Condorcet criterion.
\end{proposition}
\begin{proof}
Let $x\in V$ be the Condorcet winner for $C$. 
We prove that for every subset $S\subseteq V$ such that $x\in S$  {\tt RBorda}$(C[S],S)$ in Algorithm~\ref{alg:rborda} 
returns a permutation $\pi\colon\{1,\ldots,|S|\}\to S$ such that $\pi(1)=x$ by induction on $|S|$. 
\par
First, note that  for every subset $S\subseteq V$with $x\in S$, $x$ is the Condorcet winner for $C[S]$. 
If $|S|\leq{\tt MAXSIZE}$, then, by the definition of the algorithm, $\pi$ is an exact solution of $\LOP(C[S])$, so that we have $\pi(1)=x$ by Proposition~\ref{pr:CC}. 
Suppose $|S|>{\tt MAXSIZE}$ and assume that  {\tt RBorda}$(C[S'],S')$ returns a permutation $\sigma'$ 
such that $\sigma'(1)=x$ for all proper subset $S'$ of $S$ with $x\in S'$. 
If we denote by $\beta_S(u)$ the Borda score of $u$ with respect to the induced submatrix $C[S]$, then 
we have $\beta_S(x)>0$ since $x$ is the Condorcet winner for $C[S]$, and hence, $x\in V_1$. By the induction hypothesis, {\tt RBorda}$(C[S_1],S_1)$ returns 
a permutation $\sigma^{(1)}$ such that $\sigma^{(1)}(1)=x$. Since $\pi=(\sigma^{(1)},\sigma^{(2)})$, 
we have $\pi(1)=x$.  
\end{proof}
\par 
Proposition~\ref{pr:condorcet-rborda} cannot be strengthened to the strong Condorcet criterion. 
A counterexample is given below.  
Let 
$V=\{y_1,y_2,y_3,z_1,z_2,z_3\}$ and define $C=(c_{uv})$ by
\begin{align}\label{eq:example}
c_{uv}=\begin{cases}
          3 &\text{if $(u,v)=(y_1,y_2),(z_1,z_2)$},\\
          1 &\text{if $(u,v)=(y_2,y_3),(y_3,y_1),(z_2,z_3),(z_3,z_1),(y_2,z_1)$},\\
          0 &\text{otherwise}.
         \end{cases}
\end{align}
The strong components of $G(C)$ are $Y=\{y_1,y_2,y_3\}$ and $Z=\{z_1,z_2,z_3\}$, and we have 
$(y_2,z_1)\in A(C)$. However, since $\beta(y_2)=-1$ and $\beta(z_1)=1$, the recursive Borda method produces 
permutation $\pi$ such that $\pi^{-1}(z_1)<\pi^{-1}(y_2)$. 
%%%%%%%%%% 4th %%%%%%%%%%
%%%%%%%%%%%%%%%%%%%%%%%%%%%%%%%%%%%%%%%%%%%%%%%%%%%%%%%%%%%%%%%%%%%%%%%%%%%%%%%%%%%%%%%%%%%%%%%%%
\section{Computational Experiments}
\label{sec:experiments}
%%%%%%%%%%%%%%%%%%%%%%%%%%%%%%%%%%%%%%%%%%%%%%%%%%%%%%%%%%%%%%%%%%%%%%%%%%%%%%%%%%%%%%%%%%%%%%%%%
This section reports computational experiments on the constructive heuristics considered in this paper, including the decomposition-based heuristics proposed in Section~\ref{sec:heuristics}.
The purpose of the experiments is twofold.
First, we compare these constructive heuristics in terms of solution quality and computation time.
Second, we examine how the choice of initial solutions affects the final outcome when local search is applied.

All algorithms were implemented in Python.
Subproblems whose sizes are at most ${\tt MAXSIZE}$ were formulated as integer programming problems and solved exactly using the Gurobi optimizer 13.0.2.
For the minimum cut method, the global minimum cut computation was carried out using the Hao--Orlin algorithm~\cite{HaoOrlin1994} implemented in the LEMON C++ library~\cite{lemon}, which was accessed from Python via a custom pybind11 binding.
The experiments were conducted on a computer with Ubuntu 24.04.4 LTS 64-bit, an Intel Core i7-12700 CPU at 2.10 GHz, and 16.0 GB of memory. The source code used in the computational experiments is publicly available at
\url{https://github.com/kazando/lop-decomposition-heuristics}. 

The benchmark instances used in this section are taken from xLOLIB~\cite{xLOLIB}.
For xLOLIB instances with $n=150$ and $n=250$, best-known objective values were taken from Garc{\'i}a, Ceberio, and Lozano~\cite{GarciaCeberioLozano2019}.
We use the values reported in the above reference as a common benchmark for reporting relative deviations, although these best-known values may have been improved in later studies.
%A smaller value of ${\rm rel.dev}$ indicates a better solution.
%%%%%%%%%%%%%%%%%%%%%%%%%%%%%%%%%%%%%%%%%%%%%%%%%%%%%%%%%%%%%%%%%%%%%%%%%%%%%%%%%%%%%%%
\subsection{Preliminary Study on the Parameter ${\tt MAXSIZE}$}
\label{subsec:maxsize}
We first examine the influence of the parameter ${\tt MAXSIZE}$ on the three decomposition-based constructive heuristics, namely the level graph method, the minimum cut method, and the recursive Borda method.
The test instances are the 39 xLOLIB instances with $n=150$.
For each method and each value of ${\tt MAXSIZE}$, we compute the average relative deviation and the average computation time over the 39 instances. 
The relative deviation is defined as
\begin{equation}
  \label{equ:rel-dev}
  {\rm rel.dev}
  =
  \frac{f_{\rm best}-f}{f_{\rm best}}\times 100,
\end{equation}
where $f$ is the objective value of the solution and $f_{\rm best}$ is the corresponding best-known objective value.

Figure~\ref{fig:maxsize-dev} shows the average relative deviation versus ${\tt MAXSIZE}$.
For all three methods, the solution quality improves as ${\tt MAXSIZE}$ increases.
Among the three methods, the recursive Borda method gives the smallest relative deviation over the tested range.

Figure~\ref{fig:maxsize-time} shows the average computation time versus ${\tt MAXSIZE}$ on a logarithmic scale.
The computation times of the level graph method and the minimum cut method increase moderately up to ${\tt MAXSIZE}=40$, but become substantially larger for ${\tt MAXSIZE}=50$ and ${\tt MAXSIZE}=60$.
The recursive Borda method remains much faster than the other two methods over the tested range, although its computation time also increases around ${\tt MAXSIZE}=40$.

Based on these observations, we fix ${\tt MAXSIZE}=40$ in the subsequent experiments.
This value is chosen as a common compromise point: it substantially improves the solution quality, especially for the recursive Borda method, while avoiding the larger computation times observed for the level graph and minimum cut methods at larger values of ${\tt MAXSIZE}$.
Although the best choice of ${\tt MAXSIZE}$ may depend on both the method and the instance size, we use the same value ${\tt MAXSIZE}=40$ for all methods in order to ensure a uniform and fair comparison.
\begin{figure}[htbp]
  \centering
  \includegraphics[width=0.8\linewidth]{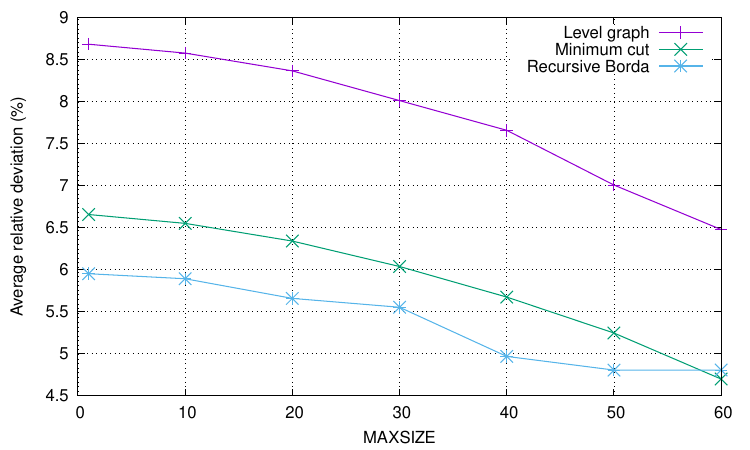}
  \caption{Average relative deviation versus ${\tt MAXSIZE}$ on xLOLIB instances with $n=150$.}
  \label{fig:maxsize-dev}
\end{figure}

\begin{figure}[htbp]
  \centering
  \includegraphics[width=0.8\linewidth]{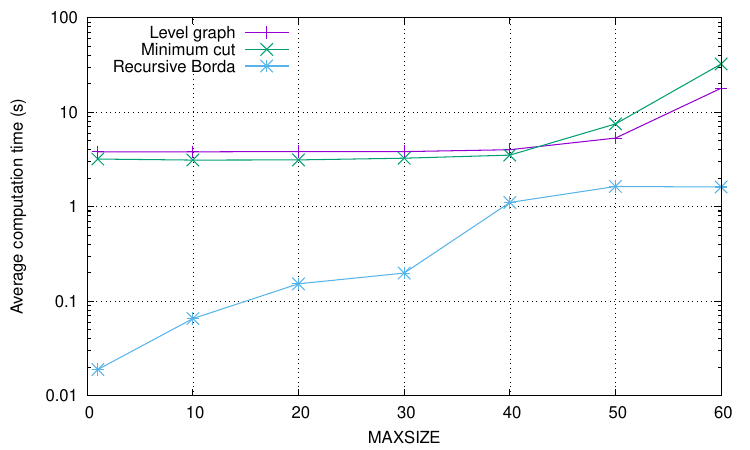}
  \caption{Average computation time versus ${\tt MAXSIZE}$ on xLOLIB instances with $n=150$ (log scale).}
  \label{fig:maxsize-time}
\end{figure}
%%%%%%%%%%%%%%%%%%%%%%%%%%%%%%%%%%%%%%%%%%%%%%%%%%%%%%%%%%%%%%%%%%%%%%%%%%%%%%%%%%%%%%%
\subsection{Comparison of Constructive Heuristics}
\label{subsec:comp_constructive}
We next compare the performance of five constructive heuristics: Becker's algorithm, the Borda rule, the level graph method, the minimum cut method, and the recursive Borda method.
The results for xLOLIB instances with $n=150$ and $n=250$ are summarized in Table~\ref{tab:constructive}.
All values are averages over 39 instances.

\begin{table}[tb]
\caption{Comparison of constructive heuristics on xLOLIB instances: average relative deviation, average rank, and average total time.}
\label{tab:constructive}
\centering
\begin{tabular}{llrrr}
\hline
Size & Method
& Avg.\ rel.dev.\ (\%)
& Avg.\ rank
& Avg.\ time (s)
\\
\hline
\multirow{5}{*}{$n=150$}
& Becker          & 7.710  & 3.513 & 0.003 \\
& Borda           & 17.358 & 5.000 & \textbf{0.001} \\
& Level graph     & 7.659  & 3.385 & 3.943 \\
& Minimum cut     & 5.674  & 1.718 & 3.493 \\
& Recursive Borda & \textbf{4.967} & \textbf{1.385} & 1.102 \\
\hline
\multirow{5}{*}{$n=250$}
& Becker          & 7.749  & 2.769 & 0.008 \\
& Borda           & 16.396 & 5.000 & \textbf{0.002} \\
& Level graph     & 9.228  & 3.897 & 15.558 \\
& Minimum cut     & 7.132  & 2.282 & 18.014 \\
& Recursive Borda & \textbf{5.306} & \textbf{1.051} & 1.093 \\
\hline
\end{tabular}
\end{table}

Table~\ref{tab:constructive} shows that the recursive Borda method gives the best average solution quality for both $n=150$ and $n=250$.
It also achieves the best average rank among the five methods for both instance sizes.
The minimum cut method is the second-best constructive heuristic in terms of both average relative deviation and average rank, but it is consistently slower than the recursive Borda method.
The level graph method performs similarly to Becker's algorithm for $n=150$, but becomes less competitive for $n=250$.
The Borda rule is the fastest among the five methods, but its solution quality is substantially worse than those of the other methods.

These results indicate that, as a standalone constructive heuristic, the recursive Borda method provides a favorable balance between solution quality and computation time among the methods considered in this paper.
%%%%%%%%%%%%%%%%%%%%%%%%%%%%%%%%%%%%%%%%%%%%%%%%%%%%%%%%%%%%%%%%%%%%%%%%%%%%%%%%%%%%%%%
We next examine the effect of different initial solutions when local search is applied.
The local search used in our experiments is based on the insertion neighborhood described in Section~\ref{subsec:local-search}.

We first consider xLOLIB instances with $n=150$.
For these instances, best-known objective values are available from Garc{\'i}a, Ceberio, and Lozano~\cite{GarciaCeberioLozano2019}, and we report average relative deviations from those values.
The results are shown in Table~\ref{tab:ls150}.

\begin{table}[htbp]
\centering
\caption{Local search on xLOLIB ($n=150$): average relative deviation and average total time over 39 instances.}
\label{tab:ls150}
\begin{tabular}{lrrr}
\hline
Method & Avg. rel.dev. (\%) & Std. dev. & Avg. total time (s) \\
\hline
Becker+LS  & 1.911 & 0.419 & \textbf{0.166} \\
Borda+LS   & \textbf{1.882} & 0.380 & 0.215 \\
rBorda+LS  & 2.071 & 0.307 & 1.209 \\
\hline
\end{tabular}
\end{table}

Table~\ref{tab:ls150} shows that, once local search is applied, the differences among the initial solutions become much smaller than in the constructive-only experiments.
In fact, Borda+LS gives the best average relative deviation on these instances, followed closely by Becker+LS, while rBorda+LS is slightly worse on average.
This indicates that, for $n=150$, the local improvement phase is able to absorb most of the advantage of a better constructive heuristic.
%In other words, although the recursive Borda method provides the best constructive solutions, this advantage does not %directly translate into better final local optima once local search is applied.

We also test larger xLOLIB instances with $n=750$ and $n=1000$.
For these instances, we do not use external best-known values.
For each instance $i$ and each method, let $f_i^{\text{method}}$ denote
the final objective value obtained by applying local search
to the initial solution generated by that method.
For brevity, we write rBorda for the recursive Borda method.
We then define
\[
f_i^{\max}
=
\max\{f_i^{\text{Becker}}, f_i^{\text{Borda}}, f_i^{\text{rBorda}}\},
\]
and measure the normalized gap of each method by
\[
\frac{f_i^{\max} - f_i^{\text{method}}}{f_i^{\max}} \times 100.
\]
We also report the average rank and the average total computation time.
The results are shown in Table~\ref{tab:ls_large}.

\begin{table}[htbp]
\centering
\caption{Local search on larger xLOLIB instances: average normalized gap to the best of the tested methods, average rank, and average total time.}
\label{tab:ls_large}
\begin{tabular}{llrrr}
\hline
Size & Method & Avg. gap (\%) & Avg. rank & Avg. total time (s) \\
\hline
\multirow{3}{*}{$n=750$}
& Becker + LS & 0.213 & 2.60 & \textbf{8.402} \\
& Borda+LS    & \textbf{0.047} & \textbf{1.52} & 10.207 \\
& rBorda+LS   & 0.108 & 1.88 & 9.616 \\
\hline
\multirow{3}{*}{$n=1000$}
& Becker + LS & 0.259 & 2.68 & \textbf{19.926} \\
& Borda+LS    & \textbf{0.064} & \textbf{1.62} & 24.665 \\
& rBorda+LS   & 0.075 & 1.70 & 21.632 \\
\hline
\end{tabular}
\end{table}

Table~\ref{tab:ls_large} shows a similar tendency on larger instances.
For both $n=750$ and $n=1000$, Becker+LS is the fastest method, whereas Borda+LS gives the best final solutions on average in terms of both average normalized gap and average rank.
The recursive Borda initialization also remains competitive, especially for $n=1000$, where its average normalized gap is close to that of Borda+LS, although it does not improve upon Borda+LS after local search.

Taken together, these results suggest that the advantage of a constructive heuristic can change substantially after local search is applied.
The recursive Borda method is effective as a standalone constructive heuristic, but its advantage over simpler score-based initial solutions becomes much less pronounced when followed by insertion-based local search.
This indicates that the local search procedure substantially reduces the differences among the tested initial-solution generators.
%%%%%%%%%%%%%%%%%%%%%%%%%%%%%%%%%%%%%%%%%%%%%%%%%%%%%%%%%%%%%%%%%%%%%%%%%%%%%%%%%
\section{Concluding Remarks}\label{sec:conclusion}
%%%%%%%%%%%%%%%%%%%%%%%%%%%%%%%%%%%%%%%%%%%%%%%%%%%%%%%%%%%%%%%%%%%%%%%%%%%%%%%%%
In this paper, we studied decomposition-based constructive heuristics for the linear ordering problem.
The proposed methods follow a common recursive framework: the index set is heuristically partitioned into smaller subsets, subproblems whose sizes are at most a prescribed threshold are solved exactly, and the resulting permutations are concatenated.
Within this framework, we considered three partitioning rules: the level graph method, the minimum cut method, and the recursive Borda method.

We also established Condorcet-type properties of the proposed methods.
The recursive Borda method satisfies the Condorcet criterion, whereas the static and sequential Borda rules do not.
The level graph method and the minimum cut method satisfy a stronger Condorcet-type property, which the recursive Borda method does not satisfy.
Thus, although the proposed methods are heuristic, they are consistent with Condorcet-type ordering properties satisfied by optimal solutions of the LOP.

Computational experiments on xLOLIB instances showed that, among the five constructive heuristics considered in this paper---Becker's algorithm, the Borda rule, the level graph method, the minimum cut method, and the recursive Borda method---the recursive Borda method provided a favorable balance between solution quality and computation time.
In particular, it outperformed Becker's algorithm, the Borda rule, and the level graph method in terms of solution quality, while also giving better solutions than the minimum cut method with substantially smaller computation time.
These results indicate that the recursive Borda method is the most promising constructive principle among those considered in this study.

We also investigated the effect of applying local search to initial solutions generated by different constructive heuristics.
The results showed that, once local search is applied, the differences among initial solutions become much smaller than in the constructive-only experiments.
Thus, the main advantage of the recursive Borda method lies in its effectiveness as a standalone constructive heuristic and as a generator of high-quality initial solutions. 
High-quality constructive heuristics remain important in practice when only limited local improvement can be afforded or when a good solution is needed quickly.

There remain several directions for future research.
First, it would be worthwhile to investigate other recursive partitioning rules within the same framework.
Second, it would be interesting to derive approximation guarantees for the proposed heuristics, either for general instances or for particular classes of instances.
Third, it would be worthwhile to investigate whether the recursive decomposition framework developed in this paper can be applied to other permutation-based combinatorial optimization problems.
%%%%%%%%%%%%%%%%%%%%%%%%%%%%%%%%%%%%%%%%%%%%%%%%%%%%%%%%%%%%%%%
\section*{Acknowledgments}
%%%%%%%%%%%%%%%%%%%%%%%%%%%%%%%%%%%%%%%%%%%%%%%%%%%%%%%%%%%%%%%%
K.~Ando's work was supported by JSPS KAKENHI Grant Number 18K11180. N.~Sukegawa's work was supported by JSPS KAKENHI Grant Number 25K08185. 
%%%%%%%%%%%%%%%%%%%%%%%%%%%%%%%%%%%%%%%%%%%%%%%%%%%%%%%%%%%%%%%%%%%%%%%%%%%%%%%%%

%------ corresponding author ------------------ 
\vspace{15pt}
{\leftskip 8.5cm \parindent 0mm
Kazutoshi Ando\\
Department of Mathematical and Systems Engineering\\
Faculty of Engineering, 
Shizuoka University\\ 
Hamamatsu 432-8561, Japan\\
E-mail: \texttt{ando.kazutoshi@shizuoka.ac.jp}
\par}
%------------------------------------------------------ 
\end{document}